\documentclass[a4paper,twoside]{amsart}
\usepackage{xypic,amscd,amssymb,latexsym,srcltx}
\usepackage{dsfont}
\usepackage{amsthm}
\def\@setcopyright{}
\makeatother

\newtheorem{lemma}{Lemma}[section]
\newtheorem{proposition}[lemma]{Proposition}

\newtheorem{theorem}[lemma]{Theorem}
\newtheorem{remark}[lemma]{Remark}
\newtheorem{conjecture}[lemma]{Conjecture}
\newtheorem*{acknowledgments*}{ACKNOWLEDGMENTS}
\newtheorem{definition}[lemma]{Definition}

\newtheorem*{conj*}{Conjecture}
\newtheorem*{thm*}{Main Theorem}
\newtheorem*{remark*}{Remark}
\newtheorem*{lemma*}{Lemma}
\newtheorem*{example*}{[Example]}

\begin{document}
\begin{center}
{\Large \bf A relation between $m_{G, N}$ and the Euler characteristic of the nerve space
of some class poset of $G$
}\\

\bigskip
\end{center}

\begin{center}
Heguo Liu$^{1}$, Xingzhong Xu$^{\ast, 1, 2}$
, Jiping Zhang$^{3}$

\end{center}

\footnotetext {$^{*}$~~Date: 27/01/2017.
}
\footnotetext {
1. Department of Mathematics, Hubei University, Wuhan, 430062, China

2. Departament de Matem$\mathrm{\grave{a}}$tiques, Universitat Aut$\mathrm{\grave{o}}$noma de Barcelona, E-08193 Bellaterra,
Spain

3. School of Mathematical Sciences, Peking University, Beijing, 100871, China

Heguo Liu's E-mail: ghliu@hubu.edu.cn

Xingzhong Xu's E-mail: xuxingzhong407@mat.uab.cat, xuxingzhong407@126.com

$\ast$ Corresponding author

Jiping Zhang's E-mail: jzhang@pku.edu.cn

Supported by National 973 Project (2011CB808003) and NSFC grant (11371124, 11501183).
}

\begin{center}
\footnotesize{\emph{
}}\\
\end{center}

\title{}
\def\abstractname{\textbf{Abstract}}

\begin{abstract}\addcontentsline{toc}{section}{\bf{English Abstract}}
Let $G$ be a finite group and $N\unlhd G$ with $|G: N|=p$ for some prime $p$.
In this note, to compute $m_{G,N}$ directly, we construct a class poset
$\mathfrak{T}_{C}(G)$ of $G$ for some cyclic subgroup $C$.
And we find a relation between $m_{G,N}$ and the Euler characteristic of the nerve space
$|N(\mathfrak{T}_{C}(G))|$ (see the Theorem 1.3).  As an application, we compute $m_{S_5, A_5}=0$ directly, and get $S_5$
is a $B$-group.

\hfil\break


\end{abstract}

\maketitle

\section{\bf Introduction}

In \cite{Bo1}, Bouc proposed the following conjecture:
\begin{conjecture}\cite[Conjecture A]{Bo1} Let $G$ be a finite group. Then $\beta(G)$ is nilpotent if
and only if $G$ is nilpotent.
\end{conjecture}
Here, $\beta(G)$ a largest quotient of a finite group which is a $B$-group and the definition of $B$-group
can be found in \cite{Bo1, Bo2} or in the Section 2.
Bouc has proven the Conjecture 1.1  under the additional assumption that finite group $G$ is solvable in \cite{Bo1}.
In \cite{XZ}, Xu and Zhang consider some special cases when the finite group $G$ is not solvable. But this result relies on
the proposition of Baumann \cite{Ba}, and his proposition relies on the Conlon theorem \cite[(80.51)]{CR}.
If we want to generalize the result of \cite{XZ}, we need use the new method to compute $m_{G, N}$ directly. Here, $N$
is a normal subgroup of $G$. And the definition of $m_{G,N}$ can be fined in \cite{Bo1,Bo2} or in the Section 2.

For this aim, some attempts have been done in the paper.
We begin to compute $m_{G,N}$ when $|G:N|=p$ for some prime number.
But it is not easy to compute $m_{G,N}$ directly even if  $|G:N|=p$. When $|G:N|=p$, we have the following observation:
\begin{eqnarray*}
&~&m_{G,N}+\frac{1}{|G|}\sum_{X\leq N} |X|\mu(X, G)\\
&=&\frac{1}{|G|}\sum_{XN= G, X\leq G} |X|\mu(X, G)+\frac{1}{|G|}\sum_{X\leq N} |X|\mu(X, G)\\
&=&\frac{1}{|G|}\sum_{XN= G, X\leq G} |X|\mu(X, G)+\frac{1}{|G|}\sum_{XN\neq G, X\leq G} |X|\mu(X, G)\\
&=&\frac{1}{|G|}\sum_{X\leq G} |X|\mu(X, G)\\
&=&m_{G,G}=0, \mathrm{if}~ G \mathrm{~is ~not~ cyclic}.
\end{eqnarray*}
So to compute $m_{G,N}$, we can compute $\frac{1}{|G|}\sum_{X\leq N} |X|\mu(X, G)$ first.
And we find there is a formula (see the proof of the Proposition 4.1) about
\begin{eqnarray*}
&~&\frac{1}{|G|}\sum_{X\leq N} |X|\mu(X, G)+\frac{1}{|G|}\sum_{X\leq N} |X|\mu(X, N)\\
&(=&\frac{1}{|G|}\sum_{X\leq N} |X|\mu(X, G)+m_{N,N}).
\end{eqnarray*}
Now, we set
$$m_{G,N}':=\frac{1}{|G|}\sum_{XN\neq G, X\leq G} |X|\mu(X, G)=\frac{1}{|G|}\sum_{X\leq N} |X|\mu(X, G);$$
and set
$$M_{G,N}':=\sum_{X\leq N} |X|\mu(X, G)=|G|m_{G,N}'.$$
We get the following theorem about $M_{G,N}'$.

\begin{proposition}Let $G$ be a finite group and  $N\unlhd G$ such that $|G:N|=p$ for some prime number $p$.
Then
$$M_{G,N}'=-\sum_{C\leq N, ~C~ \mathrm{is~ cyclic}}\varphi(|C|)-\sum_{Y\lneq G,Y\nleq N}M'_{Y, Y\cap N}.$$
Here, $\varphi$ is the Euler totient function.
\end{proposition}

To compute $M_{G,N}'$, we need compute $M_{Y, Y\cap N}'$ for every $Y\lneq G$. Since $Y\lneq G$,
thus we can get $M_{G,N}'$ by finite steps. Here, we define a new class poset of subgroups of $G$ as following:
 Let $C$ be a cyclic subgroup of $N$, define
$$\mathfrak{T}_{C}(G):=\{X|C\leq X\lneq G, X\nleq N\}.$$
We can see that $\mathfrak{T}_{C}(G)$ is a poset ordered by inclusion. We can consider poset $\mathfrak{T}_{C}(G)$
as a category with one morphism $Y\rightarrow Z$ if $Y$ is a subgroup of $Z$. We set $N(\mathfrak{T}_{C}(G))$
is the nerve of the category $\mathfrak{T}_{C}(G)$ and $|N(\mathfrak{T}_{C}(G))|$ is  the geometric realization
of $N(\mathfrak{T}_{C}(G))$.  Then we get a following computation about $m_{G,N}$.

\begin{theorem}Let $G$ be a finite group and $G$ not cyclic. Let $N\unlhd G$ such that $|G:N|=p$ for some prime number $p$.
Then
\begin{eqnarray*}
m_{G,N}
&=&\frac{1}{|G|}(\sum_{C\leq N, ~C~ \mathrm{is~ cyclic}}(1-\chi(|N(\mathfrak{T}_{C}(G))|)\cdot \varphi(|C|))).
\end{eqnarray*}
Here, $|N(\mathfrak{T}_{C}(G))|$ is a simplicial complex associated to the poset $\mathfrak{T}_{C}(G)$, and
$\chi(|N(\mathfrak{T}_{C}(G))|)$ is the Euler characteristic of the space $|N(\mathfrak{T}_{C}(G))|$.
\end{theorem}
\begin{proof} Since $m_{G,N}+m_{G,N}'=m_{G,G}=0$ when $G$ is not cyclic, thus
we prove this theorem by using the Proposition 5.4
\end{proof}

By using the method of the Theorem 1.2-3, we compute $m_{S_5, A_5}$ directly, and we get the following result.

\begin{proposition} $S_5$ is a $B$-group.
\end{proposition}
In fact, \cite{Bo3} had proved that $S_n$ is a $B$-group. And by using \cite{Ba}, we also get that $S_n$
is a $B$-group when $n\geq 5$.

After recalling the basic definitions and properties of $B$-groups in the Section 2, we prove some
lemmas about M$\mathrm{\ddot{o}}$bius function in the Section 3. And this lemmas will be used
 in the Section 4 to prove the Proposition 1.2.
In the Section 5, we construct a class poset
$\mathfrak{T}_{C}(G)$ of $G$ for some cyclic subgroup of $C$ and prove the Theorem 1.3.
As an application, we compute $m_{S_5, A_5}=0$ and get $S_5$
is a $B$-group in the Section 6.

\section{\bf Burnside rings and $B$-groups}

In this section we collect some known results that will be needed later.  For the background theory of Burnside rings and
$B$-groups, we refer to \cite{Bo1}, \cite{Bo2}.

\begin{definition}\cite[Notation 5.2.2]{Bo2} Let G be a finite group and $N\unlhd G$. Denote by $m_{G,N}$ the rational number defined by:
$$m_{G,N}=\frac{1}{|G|}\sum_{XN=G} |X|\mu(X, G),$$
where$\mu$ is the $M\ddot{o}bius$ function of the poset of subgroups of $G$.
\end{definition}

\begin{remark} If $N=1$, we have
$$m_{G,1}=\frac{1}{|G|}\sum_{X1=G} |X|\mu(X, G)=\frac{1}{|G|}|G|\mu(G, G)=1\neq 0.$$
\end{remark}

\begin{definition}\cite[Definiton 2.2]{Bo1} The finite group $G$ is called a $B$-group if
$m_{G,N}=0$ for any non-trivial normal subgroup $N$ of $G$.
\end{definition}

\begin{proposition}\cite[Proposition 5.4.10]{Bo2} Let $G$ be a finite group. If $N_{1}, N_{2}\unlhd G$
are maximal such that $m_{G,N}\neq 0$, then $G/N_{1}\cong G/N_{2}$.
\end{proposition}

\begin{definition}\cite[Notation 2.3]{Bo1} When $G$ is a finite group, and $N\unlhd G$
is maximal such that $m_{G,N}\neq 0$, set $\beta(G)=G/N.$
\end{definition}

\begin{theorem}\cite[Theorem 5.4.11]{Bo2} Let $G$ be a finite group.

1. $\beta(G)$ is a $B$-group.

2. If a $B$-group $H$ is isomorphic to a quotient of $G$, then $H$ is isomorphic to a quotient of $\beta(G)$.

3. Let $M\unlhd G$. The following conditions are equivalent:

\quad\quad (a) $m_{G,N}\neq 0$.

\quad\quad (b) The group $\beta(G)$ is isomorphic to a quotient of $G/M$.

\quad\quad (c) $\beta(G)\cong \beta(G/N)$.
\end{theorem}

We collect some properties of $m_{G,N}$ that will be needed later.

\begin{proposition}\cite[Proposition 2.5]{Bo1} Let $G$ be a finite group.
Then $G$ is a $B$-group if and only if $m_{G,N}=0$ for any minimal (non-trivial) normal subgroup of $G$.
\end{proposition}

\begin{proposition}\cite[Proposition 5.6.1]{Bo2} Let $G$ be a finite group. Then
$m_{G,G}=0$ if and only if $G$ is not cyclic.
If $P$ be cyclic of order $p$ and $p$ be a prime number,
then $m_{P, P}=\frac{p-1}{p}$.
\end{proposition}

\begin{remark} If $G$ is a finite simple group, then $G$ is a $B$-group if and only if
$G$ is not abelian.
\end{remark}


\begin{proposition}\cite[Proposition 5.3.1]{Bo2} Let $G$ be a finite group. If
$M$ and $N$ are normal subgroup of $G$ with $N\leq M$, then
$$m_{G,M}=m_{G,N}m_{G/N, M/N}.$$
\end{proposition}

%
%
%
%
%

%
%
%
%
We collect two results that will be needed later.

When $p$ is a prime number, recall that a finite group $G$ is called cyclic modulo $p$ (or $p$-hypo-elementary) if
$G/O_{p}(G)$ is cyclic. And M. Baumann has proven the Conjecture under the additional assumption that finite group $G$ is cyclic modulo $p$ in \cite{Ba}.

\begin{theorem}\cite[Theorem 3]{Ba} Let $p$ be a prime number and $G$ be a finite group. Then $\beta(G)$ is cyclic modulo $p$ if
and only if $G$ is cyclic modulo $p$.
\end{theorem}

In \cite{Bo1}, S. Bouc has proven the Conjecture under the additional assumption that finite group $G$ is solvable.

\begin{theorem}\cite[Theorem 3.1]{Bo1} Let $G$ be a solvable finite group. Then $\beta(G)$ is nilpotent if
and only if $G$ is nilpotent.
\end{theorem}

\section{\bf Some lemmas about the  M$\mathrm{\ddot{o}}$bius function}
In this section, we prove some lemmas about the M$\mathrm{\ddot{o}}$bius function. The main lemma is the Lemma 3.4, and the Lemma 3.2-3
 are prepared for it. In fact, the Lemma 3.4
will be used in computing $m_{G,N}$ where $G$ is a finite group and $N\unlhd G$.

Let $G$ be a finite group and let $\mu$ denote the M$\mathrm{\ddot{o}}$bius function of
subgroup lattice of $G$. We refer to \cite{Ai},\cite[p.94]{Y}:

Let $K,D\leq G$, recall the Zeta function of $G$ as following:
$$\zeta(K, D)=\left\{ \begin{array}{ll}
1, &
\mbox{if}~ K\leq D;
\\[2ex] 0, &\mbox{if} ~K\nleq D.\end{array}\right.$$

Set $n:=|\{K|K\leq G\}|$, we have a $n\times n$ matrix $A$ as following:
$$A:=(\zeta(K, D))_{K,D\leq G}.$$
It is easy to find that $A$ is an invertible matrix, so there exists $A^{-1}$ such that
$$AA^{-1}=E,$$
Here, $E$ is an identity element. Recall the M$\mathrm{\ddot{o}}$bius function as following:
$$(\mu(K, D))_{K,D\leq G}=A^{-1}.$$

Now, we set the subgroup lattice of $G$ as following:
$$\{K|K\leq G\}:=\{1=K_1, K_2,\ldots, K_n=G\}$$
where $n=|\{K|K\leq G\}|$.

\begin{definition}Let $K\leq G$, there exists a proper subgroup series of $K$ as following:
$$\sigma: 1=K_1\lneq  K_2\lneq  K_3\lneq \cdots \lneq K_t=K$$
where $K_i$ are subgroups of $G$ and $K_i$ is a proper subgroup of $K_{i+1}$ for all $i$.
Set $\mathfrak{X}_{K}$ is the set of elements like above $\sigma$. And we call $t$ is the length of $\sigma$ and set $t:=l(\sigma)$

We define the height of $K$ as following:
$$\mathrm{ht}(K)=\mathrm{max}\{l(\sigma)|\sigma\in \mathfrak{X}_{K}\}.$$
It is easy to see that $\mathrm{ht}(1)=1$.

\end{definition}

\begin{lemma}Let $G$ be a finite group and $K, L$ be subgroups of $G$.
If $\mathrm{ht}(K)\geq \mathrm{ht}(L)$ and $K\neq L$, then $\zeta(K, L)=0$.
\end{lemma}

\begin{proof} Suppose that $\zeta(K, L)\neq0$, by the definition of Zeta function,
 we have that $K\leq L$. Set $\mathrm{ht}(K)=t$, there exists
 a proper subgroup series of $K$ as following:
$$1=K_1\lneq  K_2\lneq  K_3\lneq \cdots \lneq K_t=K$$
Also $K\leq L$ and $K\neq L$, thus we have the following series of $L$:
$$1=K_1\lneq  K_2\lneq  K_3\lneq \cdots \lneq K_t=K\lneq L.$$
Hence $\mathrm{ht}(L)\geq t+1\gneq t= \mathrm{ht}(K)$. That is a contradiction
to $\mathrm{ht}(K)\geq \mathrm{ht}(L)$.
\end{proof}

\begin{lemma}Let $G$ be a finite group. Let $\{K_i|i=1,2,\ldots, n\}$ be the set of all subgroups of $G$. And Set $K_1=1, K_n=G$. Then we can reorder the sequence
$1=K_1, K_2,\ldots, K_n=G$ as the sequence $1=K_{l^{(1)}}, K_{l^{(2)}},\ldots, K_{l^{(n)}}=G$ such that
$(\zeta(K_{l^{(j)}}, K_{l^{(k)}}))_{n\times n}$ is an invertible upper triangular matrix.
Here, $\{l^{(1)}, l^{(2)}, \ldots, l^{(n)}\}=\{1, 2, \ldots, n\}$.
\end{lemma}

\begin{proof} As the above definition of height of subgroups, we can set
\begin{eqnarray*}
~
&~&\mathfrak{T}_{1}=\{K\leq G|\mathrm{ht}(K)=1\}=\{1=K_1\};\\
&~&\mathfrak{T}_{2}=\{K\leq G|\mathrm{ht}(K)=2\}:=\{K_{2_1}, K_{2_2},\ldots K_{2_{t_{2}}}\};\\
&~&\cdots\cdots\cdots\\
&~&\mathfrak{T}_{m-1}=\{K\leq G|\mathrm{ht}(K)=m-1\}:=\{K_{(m-1)_1}, K_{(m-1)_2},\ldots K_{(m-1)_{t_{m-1}}}\};\\
&~&\mathfrak{T}_{m}=\{K\leq G|\mathrm{ht}(K)=m\}=\{K_n=G\};\\
&~&\mathfrak{T}_{m+1}=\{K\leq G|\mathrm{ht}(K)=m+1\}=\emptyset.
\end{eqnarray*}
Now,
we can reorder
$1=K_1, K_2,\ldots, K_n=G$ as
\begin{eqnarray*}
~
&~&1=K_1;\\
&~&K_{2_1}, K_{2_2},\ldots K_{2_{t_{2}}};\\
&~&\cdots\cdots\cdots\\
&~&K_{(m-1)_1}, K_{(m-1)_2},\ldots K_{(m-1)_{t_{m-1}}};\\
&~&K_n=G.
\end{eqnarray*}
We can set
$$l^{(1)}:=1, l^{(2)}:=2_1,\ldots, (m-1)_{t_{m-1}}:=l^{(n-1)}, l^{(n)}:=n.$$
Let $s\gneq r$, we want to prove that $$\zeta(K_{l^{(s)}}, K_{l^{(r)}})=0.$$
Here, we can set
$K_{l^{(s)}}=K_{k_{a}}$ and $K_{l^{(r)}}=K_{j_{b}}$ for $1\leq a\leq t_{k},~ 1\leq b \leq t_{j}$.
Since $s\gneq r$, thus $k\geq j$.
Then by the Lemma 3.2, we have
$$\zeta(K_{k_{a}}, K_{j_{b}})=0$$
for $1\leq a\leq t_{k},~ 1\leq b \leq t_{j}$.

So, we  reorder
$1=K_1, K_2,\ldots, K_n=G$ as $1=K_{l^{(1)}}, K_{l^{(2)}},\ldots,$ $ K_{l^{(n)}}=G$ and we have
$$A:=(\zeta(K_{l^{(j)}}, K_{l^{(k)}}))_{n\times n}$$ is an invertible upper triangular matrix.
\end{proof}

%

Let us list the main lemma as following, and this lemma is used to prove the Proposition 4.1 in the Section 4.

\begin{lemma}Let $G$ be a finite group. Let $\{K_i|i=1,2,\ldots, n\}$ be the set of all subgroups of $G$.
 And Set $K_1=1, K_n=G$. Then we have $\mu(K_i, K_{i'})=0$ if $K_i\nleq K_{i'}$.
\end{lemma}

\begin{proof} By the proof of the above lemma, we can suppose that
 $\{K_i|i=1,2,\ldots, n\}$ be the set of all subgroups of $G$ and set $K_1=1, K_n=G$  such that
$(\zeta(K_{j}, K_{k}))_{n\times n}$ is an invertible upper triangular matrix.
Here, $1\leq j, k\leq n $.

To prove the lemma, we need to adjust the positions of some $K_j$ such that the sequence $1=K_1, K_2,\ldots, K_n=G$
are reordered as $1=K_{1^{(1)}}, K_{2^{(1)}},\ldots, K_{n^{(1)}}=G$, and we can get that $(\zeta(K_{j^{(1)}}, K_{k^{(1)}})_{n\times n}$ is an invertible upper triangular matrix.
Here, $j^{(1)}, k^{(1)}\in \{1^{(1)},2^{(1)},\ldots,n^{(1)}\}=\{1,2,\ldots, n\}$.

First, we can set $\mathrm{ht}(K_{i})=k$ and $\mathrm{ht}(K_{i'})=j$. Moreover,
 we set $K_{i}=K_{k_{a}}$ and $K_{i'}=K_{j_{b}}$ for $1\leq a\leq t_{k},~ 1\leq b \leq t_{j}$.
If $k\gneq j$, it is easy to see that $\mu(K_{k_{a}}, K_{j_{b}})=\mu(K_i, K_{i'})=0$.
Now, we will consider the cases when $k=j$ and $k\lneq j$ as following.

\textbf{Case 1.} $k=j$. We will prove this case when $k_{a}\gneq k_{b}$ and $k_{a}\lneq k_{b}$ as following.

\textbf{Case 1.1.} If $k_{a}\gneq k_{b}=j_{b}$, we can set
\begin{eqnarray*}
~
&~&\mathfrak{T}_{1}=\{K\leq G|\mathrm{ht}(K)=1\}=\{1=K_1\};\\
&~&\mathfrak{T}_{2}=\{K\leq G|\mathrm{ht}(K)=2\}:=\{K_{2_1}, K_{2_2},\ldots K_{2_{t_{2}}}\};\\
&~&\cdots\cdots\cdots\\
&~&\mathfrak{T}_{k}=\{K\leq G|\mathrm{ht}(K)=k\}:=\{K_{k_1}, K_{k_2},\ldots K_{k_{t_{k}}}\};\\
&~&\cdots\cdots\cdots\\
&~&\mathfrak{T}_{m-1}=\{K\leq G|\mathrm{ht}(K)=m-1\}:=\{K_{(m-1)_1}, K_{(m-1)_2},\ldots K_{(m-1)_{t_{m-1}}}\};\\
&~&\mathfrak{T}_{m}=\{K\leq G|\mathrm{ht}(K)=m\}=\{K_n=G\};\\
&~&\mathfrak{T}_{m+1}=\{K\leq G|\mathrm{ht}(K)=m+1\}=\emptyset.
\end{eqnarray*}

It is easy to see that
we can reorder of
$1=K_1, K_2,\ldots, K_n=G$ as
\begin{eqnarray*}
~
&~& K_1(=1),\\
&~& K_{2_1}, K_{2_2},\ldots K_{2_{t_{2}}},\\
&~& \cdots\cdots\cdots\\
&~& K_{k_1}, K_{k_2},\ldots,  K_{k_{b}}, K_{k_{b+1}}\ldots,  K_{k_{a}}, K_{k_{a+1}}, \ldots, K_{k_{t_{k}}},\\
&~& \cdots\cdots\cdots\\
&~& K_{(m-1)_1}, K_{(m-1)_2},\ldots K_{(m-1)_{t_{m-1}}},\\
&~& K_n(=G).
\end{eqnarray*}
And we set this order of subgroup series of $G$ as
$$1=K_{r^{(1)}}, K_{r^{(2)}},\ldots, K_{r^{(n)}}=G.$$
We have
$$B:=(\zeta(K_{r^{(j)}}, K_{r^{(k)}}))_{n\times n}$$ is an invertible upper triangular matrix by the proof of the Lemma 3.2.
Hence $B^{-1}$ is also invertible upper triangular matrix, thus $\mu(K_{j_{b}}, K_{k_{a}})=0$.
That is $\mu(K_{i}, K_{i'})=0$.

\textbf{Case 1.2.} If $k_{a}\lneq k_{b}=j_{b}$, we can set
\begin{eqnarray*}
~
&~&\mathfrak{T}_{1}=\{K\leq G|\mathrm{ht}(K)=1\}=\{1=K_1\};\\
&~&\mathfrak{T}_{2}=\{K\leq G|\mathrm{ht}(K)=2\}:=\{K_{2_1}, K_{2_2},\ldots K_{2_{t_{2}}}\};\\
&~&\cdots\cdots\cdots\\
&~&\mathfrak{T}_{k}=\{K\leq G|\mathrm{ht}(K)=k\}:=\{K_{k_1}, K_{k_2},\ldots K_{k_{t_{k}}}\};\\
&~&\cdots\cdots\cdots\\
&~&\mathfrak{T}_{m-1}=\{K\leq G|\mathrm{ht}(K)=m-1\}:=\{K_{(m-1)_1}, K_{(m-1)_2},\ldots K_{(m-1)_{t_{m-1}}}\};\\
&~&\mathfrak{T}_{m}=\{K\leq G|\mathrm{ht}(K)=m\}=\{K_n=G\};\\
&~&\mathfrak{T}_{m+1}=\{K\leq G|\mathrm{ht}(K)=m+1\}=\emptyset.
\end{eqnarray*}

It is easy to see that
we can reorder
$1=K_1, K_2,\ldots, K_n=G$ as
\begin{eqnarray*}
~
&~& K_1(=1),\\
&~& K_{2_1}, K_{2_2},\ldots K_{2_{t_{2}}},\\
&~& \cdots\cdots\cdots\\
&~& K_{k_1}, K_{k_2},\ldots,  K_{k_{b}}, K_{k_{a+1}}\ldots,  K_{k_{a}}, K_{k_{b+1}}, \ldots, K_{k_{t_{k}}},\\
&~& \cdots\cdots\cdots\\
&~& K_{(m-1)_1}, K_{(m-1)_2},\ldots K_{(m-1)_{t_{m-1}}},\\
&~& K_n(=G).
\end{eqnarray*}
That is, $K_{k_{a}}$ and $K_{k_{b}}$ switch places.
And we set this order of subgroups of $G$ as
$$1=K_{r^{(1)}}, K_{r^{(2)}},\ldots, K_{r^{(n)}}=G.$$
We have
$$B:=(\zeta(K_{r^{(j)}}, K_{r^{(k)}}))_{n\times n}$$ is an invertible upper triangular matrix by the Lemma 3.2.
Hence $B^{-1}$ is also invertible upper triangular matrix, thus $\mu(K_{k_{a}}, K_{k_{b}})=0$.
That is $\mu(K_{i}, K_{i'})=0$.

\textbf{Case 2.} $k\lneq j$. We will prove this case when $j-k=1$ and $j-k\geq 2$ as following.

\textbf{Case 2.1.} If $j-k=1$, we can set
\begin{eqnarray*}
~
&~&\mathfrak{T}_{1}=\{K\leq G|\mathrm{ht}(K)=1\}=\{1=K_1\};\\
&~&\mathfrak{T}_{2}=\{K\leq G|\mathrm{ht}(K)=2\}:=\{K_{2_1}, K_{2_2},\ldots K_{2_{t_{2}}}\};\\
&~&\cdots\cdots\cdots\\
&~&\mathfrak{T}_{k}=\{K\leq G|\mathrm{ht}(K)=k\}:=\{K_{k_1}, K_{k_2},\ldots K_{k_{t_{k}}}\};\\
&~&\mathfrak{T}_{k+1}=\{K\leq G|\mathrm{ht}(K)=k+1\}:=\{K_{(k+1)_1}, K_{(k+1)_2},\ldots K_{(k+1)_{t_{k+1}}}\};\\
&~&\cdots\cdots\cdots\\
&~&\mathfrak{T}_{m-1}=\{K\leq G|\mathrm{ht}(K)=m-1\}:=\{K_{(m-1)_1}, K_{(m-1)_2},\ldots K_{(m-1)_{t_{m-1}}}\};\\
&~&\mathfrak{T}_{m}=\{K\leq G|\mathrm{ht}(K)=m\}=\{K_n=G\};\\
&~&\mathfrak{T}_{m+1}=\{K\leq G|\mathrm{ht}(K)=m+1\}=\emptyset.
\end{eqnarray*}
We can set $a=t_{k}$, that is $K_{k_{a}}=K_{k_{t_{k}}}$. And we can set
$b=1$, that is $K_{j_{b}}=K_{(k+1)_{1}}$.

Now,
we  reorder the sequence
$1=K_1, K_2,\ldots, K_n=G$ as
\begin{eqnarray*}
~
&~& K_1(=1),\\
&~& K_{2_1}, K_{2_2},\ldots K_{2_{t_{2}}},\\
&~& \cdots\cdots\cdots\\
&~& K_{k_1}, K_{k_2},\ldots, K_{k_{t_{k}-2}}, K_{k_{t_{k}-1}}, K_{(k+1)_{1}}(=K_{j_{b}}),\\
&~& K_{k_{t_{k}}}(=K_{k_{a}}),  K_{(k+1)_2}, K_{(k+1)_3},\ldots K_{(k+1)_{t_{k+1}}},\\
&~& \cdots\cdots\cdots\\
&~& K_{(m-1)_1}, K_{(m-1)_2},\ldots K_{(m-1)_{t_{m-1}}},\\
&~& K_n(=G).
\end{eqnarray*}
That is, $K_{k_{a}}$ and $K_{j_{b}}$ switch places.
And we set the above sequence as
$$1=K_{r^{(1)}}, K_{r^{(2)}},\ldots, K_{r^{(n)}}=G.$$
That is $$r^{(1)}=1, r^{(2)}=2_1,\ldots, r^{(n-1)}=(m-1)_{t_{m-1}}, r^{(n)}=n.$$
Since $K_{k_{a}}\nleq K_{j_{b}}$, we can set
$K_{j_{b}}=K_{r^{(l)}}, K_{k_{a}}=K_{r^{(l+1)}},$
we have
$B:=(\zeta(K_{r^{(j)}}, K_{r^{(j')}}))_{n\times n}$ is an invertible upper triangular matrix by the Lemma 3.2.
Hence $B^{-1}$ is also invertible upper triangular matrix, thus $\mu( K_{k_{a}}, K_{j_{b}})=0$.
That is $\mu(K_{i}, K_{i'})=0$.

\textbf{Case 2.2.} If $j-k\geq2$, set $c:=j-k$ and we have $j=k+c$. Here, we can set
\begin{eqnarray*}
~
&~&\mathfrak{T}_{1}=\{K\leq G|\mathrm{ht}(K)=1\}=\{1=K_1\};\\
&~&\mathfrak{T}_{2}=\{K\leq G|\mathrm{ht}(K)=2\}:=\{K_{2_1}, K_{2_2},\ldots K_{2_{t_{2}}}\};\\
&~&\cdots\cdots\cdots\\
&~&\mathfrak{T}_{k}=\{K\leq G|\mathrm{ht}(K)=k\}:=\{K_{k_1}, K_{k_2},\ldots K_{k_{t_{k}}}\};\\
&~&\mathfrak{T}_{k+1}=\{K\leq G|\mathrm{ht}(K)=k+1\}:=\{K_{(k+1)_1}, K_{(k+1)_2},\ldots K_{(k+1)_{t_{k+1}}}\};\\
&~&\cdots\cdots\cdots\\
&~&\mathfrak{T}_{k+c-1}=\{K\leq G|\mathrm{ht}(K)=k+c-1\}:=\{K_{(k+c-1)_1}, K_{(k+c-1)_2},\ldots K_{(k+c-1)_{t_{k+c-1}}}\};\\
&~&\mathfrak{T}_{k+c}=\{K\leq G|\mathrm{ht}(K)=k+c\}:=\{K_{(k+c)_1}, K_{(k+c)_2},\ldots K_{(k+c)_{t_{k+c}}}\};\\
&~&\cdots\cdots\cdots\\
&~&\mathfrak{T}_{m-1}=\{K\leq G|\mathrm{ht}(K)=m-1\}:=\{K_{(m-1)_1}, K_{(m-1)_2},\ldots K_{(m-1)_{t_{m-1}}}\};\\
&~&\mathfrak{T}_{m}=\{K\leq G|\mathrm{ht}(K)=m\}=\{K_n=G\};\\
&~&\mathfrak{T}_{m+1}=\{K\leq G|\mathrm{ht}(K)=m+1\}=\emptyset.
\end{eqnarray*}
We can set $a=t_{k}$, that is $K_{k_{a}}=K_{k_{t_{k}}}$. And we can set
$b=1$, that is $K_{j_{b}}=K_{(k+c)_{1}}$.
For each $k+1\leq l \leq k+c-1$, we consider $\mathfrak{T}_{l}$ and we can suppose that there exists  $1\leq s_{l} \leq t_{l}$ such that

$$ K_{l_d} \left\{ \begin{array}{ll}
\lneq K_{(k+c)_{1}}, &
\mbox{if}~ 1\leq d \leq s_{l};
\\[2ex] \nleq K_{(k+c)_{1}}, &\mbox{if}~  s_{l}+1\leq d \leq t_{l}.\end{array}\right.$$

First,
we  reorder
$1=K_1, K_2,\ldots, K_n=G$ as
\begin{eqnarray*}
~
&~& K_1(=1),\\
&~& K_{2_1}, K_{2_2},\ldots K_{2_{t_{2}}},\\
&~& \cdots\cdots\cdots\\
&~& K_{k_1}, K_{k_2},\ldots, K_{k_{t_{k}-2}}, K_{k_{t_{k}-1}},\\
&~& K_{(k+1)_1}, K_{(k+1)_2},\ldots, K_{(k+1)_{s_{k+1}}}, \\
&~& \cdots\cdots\cdots \\
&~& K_{(k+c-1)_1}, K_{(k+c-1)_2},\ldots, K_{(k+c-1)_{s_{k+c-1}}}, \\
&~& K_{(k+c)_{1}}(=K_{j_{b}}), K_{k_{t_{k}}}(=K_{k_{a}}), \\
&~& K_{(k+1)_{s_{k+1}+1}}, K_{(k+1)_{s_{k+1}+2}},\ldots, K_{(k+1)_{t_{k+1}}},\\
&~& \cdots\cdots\cdots\\
&~& K_{(k+c-1)_{s_{k+c-1}+1}}, K_{(k+1)_{s_{k+c-1}+2}},\ldots, K_{(k+1)_{t_{k+1}}},\\
&~&  K_{(k+c)_2}, K_{(k+c)_3},\ldots K_{(k+c)_{t_{k+c}}},\\
&~& \cdots\cdots\cdots\\
&~& K_{(m-1)_1}, K_{(m-1)_2},\ldots K_{(m-1)_{t_{m-1}}},\\
&~& K_n(=G).
\end{eqnarray*}

Now, we set the above sequence as
$$1=K_{r^{(1)}}, K_{r^{(2)}},\ldots, K_{r^{(n)}}=G.$$

To prove $$B:=(\zeta(K_{r^{(j)}}, K_{r^{(j')}}))_{n\times n}$$ is an invertible upper triangular matrix, we will prove the following (1)-(3) first:

(1) For each \begin{eqnarray*}
K_v
&\in&\{ K_{(k+1)_{s_{k+1}+1}}, K_{(k+1)_{s_{k+1}+2}},\ldots, K_{(k+1)_{t_{k+1}}},\\
&~& \cdots\cdots\cdots\\
&~& K_{(k+c-1)_{s_{k+c-1}+1}}, K_{(k+1)_{s_{k+c-1}+2}},\ldots, K_{(k+1)_{t_{k+1}}}\}
\end{eqnarray*} and
\begin{eqnarray*}
K_u
&\in&\{ K_{(k+1)_1}, K_{(k+1)_2},\ldots, K_{(k+1)_{s_{k+1}}}, \\
&~& \cdots\cdots\cdots \\
&~& K_{(k+c-1)_1}, K_{(k+c-1)_2},\ldots, K_{(k+c-1)_{s_{k+c-1}}}\},
\end{eqnarray*}
we can see that
$$\zeta(K_v, K_u)=0.$$
Suppose $\zeta(K_v, K_u)\neq 0$, we have $K_v\leq K_u$.
But $K_u\lneq K_{(k+c)_{1}}$ and $K_v\nleq K_{(k+c)_{1}}$, that is a contradiction.
So $\zeta(K_v, K_u)=0.$

(2) For each \begin{eqnarray*}
K_v
&\in&\{ K_{(k+1)_{s_{k+1}+1}}, K_{(k+1)_{s_{k+1}+2}},\ldots, K_{(k+1)_{t_{k+1}}},\\
&~& \cdots\cdots\cdots\\
&~& K_{(k+c-1)_{s_{k+c-1}+1}}, K_{(k+1)_{s_{k+c-1}+2}},\ldots, K_{(k+1)_{t_{k+1}}}\},
\end{eqnarray*}
we have $K_v\nleq K_{(k+c)_{1}}$, thus
$$\zeta(K_v, K_{(k+c)_{1}})=0.$$

(3) For each\begin{eqnarray*}
K_u
&\in&\{ K_{(k+1)_1}, K_{(k+1)_2},\ldots, K_{(k+1)_{s_{k+1}}}, \\
&~& \cdots\cdots\cdots \\
&~& K_{(k+c-1)_1}, K_{(k+c-1)_2},\ldots, K_{(k+c-1)_{s_{k+c-1}}}\},
\end{eqnarray*}
we can see that
$$\zeta(K_{k_{t_{k}}}, K_u)=0.$$
Suppose $\zeta(K_{k_{t_{k}}}, K_u)\neq 0$, that is $K_{k_{t_{k}}}\leq K_u$.
But $K_u\lneq K_{(k+c)_{1}}$, thus $K_{k_{t_{k}}}\lneq K_{(k+c)_{1}}$.
And we know $K_{(k+c)_{1}}=K_{j_{b}}, K_{k_{t_{k}}}=K_{k_{a}}$
and $K_{k_{a}}=K_i\nleq K_j=K_{j_{b}}$. That is a contradiction.
So $\zeta(K_{k_{t_{k}}}, K_u)=0.$

Since $(1)-(3)$ hold, we have $$B:=(\zeta(K_{r^{(j)}}, K_{r^{(j')}}))_{n\times n}$$ is an invertible upper triangular matrix.
It implies $B^{-1}$ is also an invertible upper triangular matrix.
We can set
$$K_{j_{b}}=K_{r^{(l)}}, K_{k_{a}}=K_{r^{(l+1)}}.$$
That is $\mu( K_{r^{(l+1)}}, K_{r^{(l)}})=0$ because $B^{-1}$ is an invertible upper triangular matrix..
So $\mu( K_{k_{a}}, K_{j_{b}})=0$, it implies $\mu(K_{i}, K_{i'})=0$.
\end{proof}

\section{\bf Computing the $m_{G, N}$ when $|G:N|=p$ for some prime number $p$}

Let $G$ be a finite group and $N \unlhd G$ with $|G: N|=p$. In this section we will compute $m_{G, N}$.
First, we set
$$m_{G,N}':=\frac{1}{|G|}\sum_{XN\neq G, X\leq G} |X|\mu(X, G)=\frac{1}{|G|}\sum_{X\leq N} |X|\mu(X, G);$$
and set
$$M_{G,N}':=\sum_{X\leq N} |X|\mu(X, G)=|G|m_{G,N}'.$$
We can see that
\begin{eqnarray*}
m_{G,N}+m_{G,N}'
&=&\frac{1}{|G|}\sum_{XN= G, X\leq G} |X|\mu(X, G)+\frac{1}{|G|}\sum_{XN\neq G, X\leq G} |X|\mu(X, G)\\
&=&\frac{1}{|G|}\sum_{X\leq G} |X|\mu(X, G)\\
&=&m_{G,G}.
\end{eqnarray*}

So to compute $m_{G, N}$, we only need to compute $M'_{G,N}$. Thus we have the following propositions.
\begin{proposition}Let $G$ be a finite group and $N\unlhd G$ such that $|G:N|=p$ for some prime number $p$.
Then
$$
M_{G,N}'
=-\sum_{Y\lneq G}\sum_{X\leq N\cap Y} |X|\mu(X, Y)$$
\end{proposition}

\begin{proof}
Let $\{1=X_{1}, X_{2},\ldots, X_{n}\}$ be the poset of subgroups of $G$ and we set
$X_{n-1}=N, X_{n}=G$. We have the follow matrix:
$$A:=(\zeta(X_{i},X_{j})):=\left(%
\begin{array}{ccc}
A_{1} & \alpha & \beta \\
0 & 1 & 1\\
0 & 0 & 1\\
\end{array}%
\right)$$
where $A_{1}:=(\zeta(X_{i},X_{j}))_{i,j\leq n-2}$ and
$$\alpha:=\left(%
\begin{array}{c}
\zeta(X_{1},N) \\
\zeta(X_{2},N)\\
\vdots\\
\zeta(X_{n-2},N) \\
\end{array}%
\right),~
\beta:=\left(%
\begin{array}{c}
\zeta(X_{1},G) \\
\zeta(X_{2},G)\\
\vdots\\
\zeta(X_{n-2},G) \\
\end{array}%
\right)=\left(%
\begin{array}{c}
1 \\
1\\
\vdots\\
1 \\
\end{array}%
\right)_{(n-2)\times 1}.
$$

Here, we set $B:=(\mu(X_{i},X_{j}))$. We know that $B=A^{-1}$.
We can set
$$B:=(\mu(X_{i},X_{j})):=\left(%
\begin{array}{ccc}
B_{1} & \gamma & \delta \\
0 & 1 & -1\\
0 & 0 & 1\\
\end{array}%
\right)$$
where $B_{1}:=(\mu(X_{i},X_{j}))_{i,j\leq n-2}$ and
$$\gamma:=\left(%
\begin{array}{c}
\mu(X_{1},N) \\
\mu(X_{2},N)\\
\vdots\\
\mu(X_{n-2},N) \\
\end{array}%
\right),~
\delta:=\left(%
\begin{array}{c}
\mu(X_{1},G) \\
\mu(X_{2},G)\\
\vdots\\
\mu(X_{n-2},G) \\
\end{array}%
\right).
$$
Since $AB=1$, we have $A_{1}B_{1}=1$,
$$A_{1}\gamma+\alpha=0,$$
and, $$A_{1}\delta-\alpha+\beta=0.$$
So we have
$A_{1}(\gamma+\delta)=-\beta,$ that is $\gamma+\delta=-A_{1}^{-1}\beta=-B_{1}\beta.$

Now
we compute the following:
\begin{eqnarray*}
M_{G,N}'+|N|m_{N,N}
&=&\sum_{X\leq N} |X|\mu(X, G)+\sum_{X\leq N} |X|\mu(X, N)\\
&=&\sum_{X\leq N} |X|(\mu(X, N)+\mu(X, G))\\
&=&\sum_{X\lneq N} |X|(\mu(X, N)+\mu(X, G))+(\mu(N, N)+\mu(N,G))\\
&=&\sum_{X\lneq N} |X|(\mu(X, N)+\mu(X, G)).
\end{eqnarray*}
Here, $\mu(N, N)=1, \mu(N,G)=-1$ because $N$ is a maximal subgroup of $G$ by $|G/N|=p$ for some prime number $p$.
Since $$
\gamma+\delta=\left(%
\begin{array}{c}
\mu(X_1, N)+\mu(X_1, G) \\
\mu(X_2, N)+\mu(X_2, G)\\
\vdots\\
\mu(X_{n-2}, N)+\mu(X_{n-2}, G) \\
\end{array}%
\right),$$
and, $\gamma+\delta=-B_{1}\beta. $ It implies
$$\gamma+\delta=-\left(%
\begin{array}{cccc}
\mu(X_1, X_1)& \mu(X_1, X_2)&\cdots & \mu(X_1, X_{n-2}) \\
\mu(X_2, X_1)& \mu(X_2, X_2)&\cdots & \mu(X_2, X_{n-2}) \\
\vdots & \vdots &\ddots & \vdots\\
\mu(X_{n-2}, X_1)& \mu(X_{n-2}, X_2)&\cdots & \mu(X_{n-2}, X_{n-2}) \\
\end{array}%
\right)\cdot
\left(%
\begin{array}{c}
1 \\
1\\
\vdots\\
1 \\
\end{array}%
\right)_{(n-2)\times 1}.$$
Thus, we have
$$\mu(X_i, N)+\mu(X_i, G)= -\sum_{Y\lneq G, Y\neq N}\mu(X_i, Y)$$
for $1\leq i\leq n-2$.
Hence, we have
\begin{eqnarray*}
M_{G,N}'+|N|m_{N,N}
&=&-\sum_{X\lneq N} |X|(\mu(X, N)+\mu(X, G))\\
&=&-\sum_{Y\lneq G, Y\neq N}\sum_{X\lneq N} |X|\mu(X, Y).
\end{eqnarray*}

We can see
\begin{eqnarray*}
&~&\sum_{Y\lneq G}\sum_{X\leq N} |X|\mu(X, Y)\\
&=&(\sum_{Y\lneq G, Y\neq N}\sum_{X\lneq N} |X|\mu(X, Y))+\sum_{Y\lneq G, Y\neq N}|N|\mu(N, Y)+\sum_{1\leq X\leq N}|X|\mu(X, N).
\end{eqnarray*}
By the Lemma 3.4, we have $\mu(N, Y)=0$ when $N\nleq Y$. And since
$$\sum_{1\leq X\leq N}|X|\mu(X, N)=|N|m_{N,N},$$
we have $$
\sum_{Y\lneq G}\sum_{X\leq N} |X|\mu(X, Y)=
\sum_{Y\lneq G, Y\neq N}\sum_{X\lneq N} |X|\mu(X, Y)+|N|m_{N,N}.$$
Hence,
\begin{eqnarray*}
M_{G,N}'+|N|m_{N,N}
&=&-\sum_{X\lneq N} |X|(\mu(X, N)+\mu(X, G))\\
&=&-\sum_{Y\lneq G, Y\neq N}\sum_{X\lneq N} |X|\mu(X, Y)\\
&=&-(\sum_{Y\lneq G}\sum_{X\leq N} |X|\mu(X, Y)-|N|m_{N,N}).
\end{eqnarray*}
So we have
$$M_{G,N}'=-\sum_{Y\lneq G}\sum_{X\leq N} |X|\mu(X, Y).$$

Also $\mu(X, Y)=0$ if $X\nleq Y$, thus we have
\begin{eqnarray*}
M_{G,N}'
&=&-\sum_{Y\lneq G}\sum_{X\leq N} |X|\mu(X, Y)\\
&=&-\sum_{Y\lneq G}\sum_{X\leq N\cap Y} |X|\mu(X, Y).
\end{eqnarray*}
\end{proof}

\begin{proposition}Let $G$ be a finite group and $N\unlhd G$ such that $|G:N|=p$ for some prime number $p$.
Then
$$M_{G,N}'=-\sum_{C\leq N, ~C~ \mathrm{is~ cyclic}}\varphi(|C|)-\sum_{Y\lneq G,Y\nleq N}M'_{Y, Y\cap N}.$$
\end{proposition}

\begin{proof}By the Proposition 4.1, we have
$$M_{G,N}'
=-\sum_{Y\lneq G}\sum_{X\leq N\cap Y} |X|\mu(X, Y)$$

We will compute $\sum_{X\leq N\cap Y} |X|\mu(X, Y)$
by considering the cases when $Y\leq N$ and $Y\neq N$ in the following.

\textbf{Case 1.} $Y\leq N$. We have
\begin{eqnarray*}
\sum_{X\leq N\cap Y} |X|\mu(X, Y)
&=&\sum_{X\leq Y} |X|\mu(X, Y)\\
&=&|Y|m_{Y,Y}.
\end{eqnarray*}
If $Y$ is not cyclic, we have $m_{Y,Y}=0$. If
$Y$ is cyclic, we have $m_{Y,Y}=\frac{\varphi(|Y|)}{|Y|}$.
Hence, we have $$\sum_{X\leq N\cap Y} |X|\mu(X, Y)=\left\{ \begin{array}{ll}
\varphi(|Y|), &
\mbox{if}~ Y~is~ cyclic;
\\[2ex] 0, &\mbox{if} ~Y~is~not~ cyclic.\end{array}\right.$$

\textbf{Case 2.} $Y\nleq N$. That is $YN=G$, it implies $|Y: Y\cap N|=|G: N|=p$. So we have
\begin{eqnarray*}
\sum_{X\leq N\cap Y} |X|\mu(X, Y)
&=&M_{Y, Y\cap N}'
\end{eqnarray*}
by the Definition of $M_{Y, Y\cap N}'$.

Hence,
\begin{eqnarray*}
M_{G,N}'
&=&-\sum_{Y\lneq G, Y\leq N, Y~is~cyclic~~}(\sum_{X\leq N\cap Y} |X|\mu(X, Y))\\
&~&-\sum_{Y\lneq G,Y\leq N, Y~is~not~cyclic~~}(\sum_{X\leq N\cap Y} |X|\mu(X, Y))\\
&~&-\sum_{Y\lneq G, Y\nleq N}\sum_{X\leq N\cap Y} |X|\mu(X, Y)\\
&=&-\sum_{Y\lneq G, Y\leq N, Y~is~cyclic} \varphi(|Y|)\\
&~&-\sum_{Y\lneq G,Y\leq N, Y~is~not~cyclic} 0\\
&~&-\sum_{Y\lneq G, Y\nleq N}M_{Y, Y\cap N}'\\
&=&-\sum_{C\leq N, ~C~ \mathrm{is~ cyclic}}\varphi(|C|)-\sum_{Y\lneq G,Y\nleq N}M'_{Y, Y\cap N}.
\end{eqnarray*}
\end{proof}

\begin{remark} (1) When $Y\nleq N$,
we have $YN=G$ because $|G:N|=p$ for some prime number $p$.
It implies that $|Y: Y\cap N|=|Y/(Y\cap N)|=|YN/N|=|G/N|=p$, we can repeat the operations on $M'_{Y, Y\cap N}$ as the Proposition 4.1.

(2) To compute $M_{G,N}'$, we need compute $M_{Y, Y\cap N}'$ for every $Y\lneq G$. Since $Y\lneq G$,
thus we can get $M_{G,N}'$ by finite steps.
\end{remark}

\begin{remark}It is easy to see that
$$\sum_{C\leq N, ~C~ \mathrm{is~ cyclic}}\varphi(|C|)=|N|.$$
\end{remark}

\section{\bf A class poset of subgroups of $G$}

To compute $M_{Y, Y\cap N}'$ for every $Y\lneq G$, we define a new class poset of subgroups of $G$ in this section.
And we find the relation between $M_{G,N}'$ and this class poset.

\begin{definition} Let $G$ be a finite group and $N\unlhd G$. Let $C$ be a cyclic subgroup of $N$, define
$$\mathfrak{T}_{C}(G):=\{X|C\leq X\lneq G, X\nleq N\}.$$
We can see that $\mathfrak{T}_{C}(G)$ is a poset ordered by inclusion. We can consider poset $\mathfrak{T}_{C}(G)$
as a category with one morphism $Y\rightarrow Z$ if $Y$ is a subgroup of $Z$. We set $N(\mathfrak{T}_{C}(G))$
is the nerve of the category $\mathfrak{T}_{C}(G)$ and $|N(\mathfrak{T}_{C}(G))|$ is  the geometric realization
of $N(\mathfrak{T}_{C}(G))$. More detail of topology can be seen in \cite{DH}.
\end{definition}

\begin{remark}
Let $A,B,D\in \mathfrak{T}_{C}(G)$, if
$A\leq B, A\leq D$, then $B\cap D\in \mathfrak{T}_{C}(G)$.
\end{remark}

Since we use the Euler characteristic of $|N(\mathfrak{T}_{C}(G))|$ in Proposition 5.4, thus
we recall the definition of the Euler characteristic as following:
\begin{definition}\cite[\S 22]{M} The Euler characteristic (or Euler number) of a finite complex $K$ is defined, classically,
by the equation
$$\chi(K) =\sum_{i}(-1)^i \mathrm{rank}(C_i(K)).$$
Said differently, $\chi(K)$ is the alternating sum of the number of simplices of $K$ in
each dimension.
\end{definition}

\begin{proposition}Let $G$ be a finite group and $N\unlhd G$ such that $|G:N|=p$ for some prime number $p$.
Then
\begin{eqnarray*}
M_{G,N}'
&=&-\sum_{C\leq N, ~C~ \mathrm{is~ cyclic}}\varphi(|C|)-\sum_{Y\lneq G,Y\nleq N}M'_{Y, Y\cap N}\\
&=&-\sum_{C\leq N, ~C~ \mathrm{is~ cyclic}}(1-\chi(|N(\mathfrak{T}_{C}(G))|)\cdot \varphi(|C|)).
\end{eqnarray*}
Here, $|N(\mathfrak{T}_{C}(G))|$ is a simplicial complex associated to the poset $\mathfrak{T}_{C}(G)$, and
$\chi(|N(\mathfrak{T}_{C}(G))|)$ is the Euler characteristic of the space $|N(\mathfrak{T}_{C}(G))|$.
\end{proposition}

\begin{proof}Let $Y\lneq G$ and $Y\nleq N$, since $YN=G$, we have $Y/(Y\cap N)\cong G/N$. So by the Proposition 4.2
and the Remark 4.3, we also have
\begin{eqnarray*}
M_{Y, Y\cap N}'
&=&-\sum_{C\leq Y\cap N, ~C~ \mathrm{is~ cyclic}}\varphi(|C|)-\sum_{Y_1\lneq Y,Y_1\nleq Y\cap N}M'_{Y_1, Y_1\cap N}\\
&=&-\sum_{C\leq Y\cap N, ~C~ \mathrm{is~ cyclic}}\varphi(|C|)-\sum_{Y_1\lneq Y,Y_1\nleq N}M'_{Y_1, Y_1\cap N}.
\end{eqnarray*}
Here, we also have  $Y_{1}N=G$ because $Y_1\nleq N$ and $|G:N|=p$. We also repeat the operations of the
Proposition 4.2 on $M'_{Y_1, Y_1\cap N}$ by the Remark 4.3.
So
\begin{eqnarray*}
&~&\sum_{Y\lneq G,Y\nleq N}M'_{Y, Y\cap N}\\
&=&\sum_{Y\lneq G,Y\nleq N}(-\sum_{C\leq Y\cap N, ~C~ \mathrm{is~ cyclic}}\varphi(|C|)-\sum_{Y_1\lneq Y,Y_1\nleq N}M'_{Y_1, Y_1\cap N})\\
&=&-\sum_{Y\lneq G,Y\nleq N}(\sum_{C\leq Y\cap N, ~C~ \mathrm{is~ cyclic}}\varphi(|C|))\\
&~&-\sum_{Y\lneq G,Y\nleq N}\sum_{Y_1\lneq Y,Y_1\nleq N}M'_{Y_1, Y_1\cap N}\\
&=&-\sum_{Y\lneq G,Y\nleq N}(\sum_{C\leq Y\cap N, ~C~ \mathrm{is~ cyclic}}\varphi(|C|))\\
&~&-\sum_{Y\lneq G,Y\nleq N}\sum_{Y_1\lneq Y,Y_1\nleq N}(-\sum_{C\leq Y_1\cap N, ~C~ \mathrm{is~ cyclic}}\varphi(|C|)
-\sum_{Y_2\lneq Y_1,Y_2\nleq N}M'_{Y_2, Y_2\cap N})\\
&=&-\sum_{Y\lneq G,Y\nleq N}(\sum_{C\leq Y\cap N, ~C~ \mathrm{is~ cyclic}}\varphi(|C|))\\
&~&+\sum_{Y\lneq G,Y\nleq N}\sum_{Y_1\lneq Y,Y_1\nleq N}(\sum_{C\leq Y_1\cap N, ~C~ \mathrm{is~ cyclic}}\varphi(|C|))\\
&~&-\sum_{Y\lneq G,Y\nleq N}\sum_{Y_1\lneq Y,Y_1\nleq N}\sum_{Y_2\lneq Y_1,Y_2\nleq N}M'_{Y_2, Y_2\cap N}\\
&=&\cdots\cdots\cdots\\
&=&-\sum_{C\leq N, ~C~ \mathrm{is~ cyclic}}\sum_{i}\sum_{\sigma\in N(\mathfrak{T}_{C}(G))_{i}}(-1)^{i}\cdot \varphi(|C|))\\
&=&-\sum_{C\leq N, ~C~ \mathrm{is~ cyclic}}\chi(|N(\mathfrak{T}_{C}(G))|)\cdot \varphi(|C|)).
\end{eqnarray*}
Here, $\sigma$ is a $i$-simplex of nerve $N(\mathfrak{T}_{C}(G))$ and $\sigma$ is not degenerate.
\end{proof}

\begin{theorem}\cite{M}The Euler characteristic of a contractible space is 1.
\end{theorem}

\begin{proof}Since the Euler characteristic is a topological invariant and the contractible space
is homotopy-equivalent to a point, thus we have the Euler characteristic of a contractible space is 1.
\end{proof}

\begin{proposition}Let $G$ be a finite group and $G$ be not cyclic. Let $N\unlhd G$ such that $|G:N|=p$ for some prime number $p$.
If the space $|N(\mathfrak{T}_{C}(G))|$ is contractible for each cyclic subgroup $C$ of $N$, then
$m_{G, N}= 0$.
\end{proposition}

\begin{proof}By the Proposition 4.4, we have
\begin{eqnarray*}
M_{G,N}'
&=&-\sum_{C\leq N, ~C~ \mathrm{is~ cyclic}}\varphi(|C|)-\sum_{Y\lneq G,Y\nleq N}M'_{Y, Y\cap N}\\
&=&-\sum_{C\leq N, ~C~ \mathrm{is~ cyclic}}(1-\chi(|N(\mathfrak{T}_{C}(G))|)\cdot \varphi(|C|)).
\end{eqnarray*}
Since for each cyclic subgroup $C$ of $N$, we have $|N(\mathfrak{T}_{C}(G))|$ is contractible.
It implies $\chi(|N(\mathfrak{T}_{C}(G)))=1$ by the Theorem 5.5, thus $M_{G,N}'=0$.

By the definition of $M_{G,N}'$, we know that
$$m_{G,N}+\frac{1}{|G|}M_{G,N}'=m_{G,G}=0.$$
So $m_{G,N}=0$.
\end{proof}

Now, we post the following problems about $S_n$.
\begin{conjecture}Let $C$ be a cyclic subgroup of $A_{n}$, then the space $|N(\mathfrak{T}_{C}(S_{n}))|$
is contractible for $n\geq 5$.
\end{conjecture}

\begin{remark}If the Conjecture 5.7 holds, then we have $m_{S_n, A_n}=0$. That is $\beta(S_n)=S_n$
is not solvable for $n\geq 5$. It implies that the Conjectures 1.1-2 hold when $G$ is $S_n$ for  $n\geq 5$.
\end{remark}

\section{\bf Connected simplicial complex of a poset}

To compute $\chi(|N(\mathfrak{T}_{C}(G))|)$, we prove the following theorem. Here, the background of the topology, we
refer to \cite{Ht,M}.
\begin{proposition}Set $\mathcal{C}=\mathfrak{T}_{C}(G)$. Let $N(\mathcal{C})$ be the nerve
of category $\mathcal{C}$. If the space $|N(\mathcal{C})|$ is connected, then
$H_{0}(|N(\mathcal{C})|)\cong Z$ and $H_{i}(|N(\mathcal{C})|)=0$ for $i\geq 1$.
\end{proposition}

\begin{proof}Let $X\in \mathrm{Ob}(\mathcal{C})$, set
 $\mathcal{C}_{X}$ be the full subcategory of $\mathcal{C}$ with object set
$\{Y\in \mathrm{Ob}(\mathcal{C})|Y\leq X\}$.

Let $M_1, M_2,\ldots, M_n$ be all maximal elements of the poset $\mathcal{C}$.

\textbf{Case 1.} $n=2$. We can see that
$$|N(\mathcal{C})|\cong|N(\mathcal{C}_{M_1})|\sqcup_{|N(\mathcal{C}_{M_1\cap M_2})|}|N(\mathcal{C}_{M_2})|.$$
Here, $M_1\cap M_2\in \mathrm{Ob}(\mathcal{C})$ because $|N(\mathcal{C})|$ is connected.

We can see $|N(\mathcal{C}_{M_1})|, |N(\mathcal{C}_{M_1\cap M_2})|, |N(\mathcal{C}_{M_2})|$
are all contractible because each category has a terminal object.
Hence by\cite[Theorem 25.1]{M}, we have $H_{0}(|N(\mathcal{C})|)\cong Z$ and $H_{i}(|N(\mathcal{C})|)=0$ for $i\geq 1$.

\textbf{Case 2.} $n\gneq 2$. We define a graph $\mathcal{D}$ as follow:

(1) The vertexes
$\mathcal{D}$ of the graph are $\{M_1, M_2,\ldots, M_n\}$;

(2) There exists a edge between two vertexes $M_i, M_j$ if
and only if there exists $X\in \mathrm{Ob}(\mathcal{C})$ such that $X\leq M_i$ and $X\leq M_j$ (It implies
$M_i\cap M_j\in \mathrm{Ob}(\mathcal{C})$ by the Remark 5.2).

Since $|N(\mathcal{C})|$ is connected, thus the graph $\mathcal{D}$ is connected.
Since the graph $\mathcal{D}$ is connected, we can delete one  vertex of  $\mathcal{D}$ and remaining part of
the graph $\mathcal{D}$ is also connected. Here, we can set that this deleted vertex is $M_1$.

First, set $\mathcal{C}_{M_1}'$ be the full subcategory of $\mathcal{C}$ with object's set
$\{X\in \mathrm{Ob}(\mathcal{C})|X\leq M_i~ \mathrm{for}~ \mathrm{some}~ 2\leq i\leq n\}$.
And set $\breve{\mathcal{C}}_{M_1}$ be the full subcategory of $\mathcal{C}$ with object's set
$\{X\in \mathrm{Ob}(\mathcal{C})|X\leq M_1\}\cap\{X\in \mathrm{Ob}(\mathcal{C})|X\leq M_i~ \mathrm{for}~ \mathrm{some}~ 2\leq i\leq n\}$.
Then we have
$$|N(\mathcal{C})|\cong|N(\mathcal{C}_{M_1})|\sqcup_{|N(\breve{\mathcal{C}}_{M_1})|}|N(\mathcal{C}_{M_1}')|.$$

We can see $|N(\mathcal{C}_{M_1}')|$ is connected, it implies
$H_{0}(|N(\mathcal{C}_{M_1}')|)\cong Z$ and $H_{i}(|N(\mathcal{C}_{M_1}')|)=0$ for $i\geq 1$ by induction on $n$.

Now, we want to prove that $|N(\breve{\mathcal{C}}_{M_1})|$ is also connected.
By the definition of $\breve{\mathcal{C}}_{M_1}$, each maximal object of $\breve{\mathcal{C}}_{M_1}$
is in the set $\{M_1\cap M_2, M_1\cap M_3,\ldots M_1\cap M_n\}$.

Since $|N(\mathcal{C}_{M_1}')|$ is connected, thus we have for each $M_1\cap M_i, M_1\cap M_j$, $2\leq i\neq j\leq n$,
there exist
$$K_1, K_2,\ldots, K_s\in \{M_1\cap M_2, M_1\cap M_3,\ldots M_1\cap M_n\},~ and~$$
$$T_1,T_2,\ldots, T_{s+1}\in \mathrm{Ob}(\mathcal{C}_{M_1}')$$
such that
$$K_1=M\cap M_i, K_s=M\cap  M_j ,~\mathrm{and},~$$
$$\mathrm{either}~ T_r\leq K_r, T_r\leq K_{r+1},$$
$$\mathrm{or}~ K_r\leq T_r, K_{r+1}\leq T_r $$
for each $1\leq r\leq s$.
If $T_r\leq K_r, K_{r+1}$, that is $T_r\in \mathrm{Ob}(\breve{\mathcal{C}}_{M_1})$.
If $K_r, K_{r+1}\leq T_r$, thus $K_r\leq M_1, T_r$, it implies
$M_1\cap T_r\in \mathrm{Ob}(\mathcal{C})$. Thus
$M_1\cap T_r\in \mathrm{Ob}(\breve{\mathcal{C}}_{M_1})$ and  $K_r, K_{r+1}\leq M_1\cap T_r$.
That is $|N(\breve{\mathcal{C}}_{M_1})|$ is connected. It implies
$H_{0}(|N(\breve{\mathcal{C}}_{M_1}')|)\cong Z$ and $H_{i}(|N(\breve{\mathcal{C}}_{M_1}')|)=0$ for $i\geq 1$ by induction on $n$.
Since $|N(\mathcal{C})|\cong|N(\mathcal{C}_{M_1})|\sqcup_{|N(\breve{\mathcal{C}}_{M_1})|}|N(\mathcal{C}_{M_1}')|$,
thus we have $H_{0}(|N(\mathcal{C})|)\cong Z$ and $H_{i}(|N(\mathcal{C})|)=0$ for $i\geq 1$.
\end{proof}

\begin{remark}By the above theorem, if the space $|N(\mathcal{C})|$ is connected, then $\chi(|N(\mathcal{C})|)=1$.
\end{remark}

\section{\bf  To compute $m_{S_5, A_5}$ }

First, we list some results about symmetric group.
Let $S_{n}$ be a symmetric group of degree $n$ and $A_{n}$ be a  alternating group of degree $n$.

\begin{theorem}\cite[Appendix]{LPS, AS}Let $S_n$  act on a set $\Omega$ of size  $n$. Then  every maximal subgroup $G(\neq A_n)$ of $S_n$, is of one of the types (a)-(f)  below:

(a) $G=S_m\times S_k$,  with  $n=m+k$  and $ m\neq k$  (intransitive  case);

(b) $ G=S_m\wr S_k$, with  $n =m^k$,  $m\gneq  1$ and  $k\gneq 1$  (imprimitive
case);

(c) $G =  AGL_k(p)$, with  $n =pk$  and  $p$  prime  (affine  case);

(d)  $G = T^k \cdot(\mathrm{Out}T\times S_k)$, with  $T$  a nonabelian simple group,
$k\geq 2$ and $n=|T|^{k-1}$ (diagonalcase);

(e) $G = S_m\wr S_k$, with  $n  =  m^k$, $m\gneq 5$ and $ k\gneq 1$, excluding the
case where $X=  A_n$,  and $G$ is imprimitive on $\Omega$  (wreath  case-see Remark 2
below);

(f)  $T \unlhd G \leq \mathrm{Aut}T$,  with  $T$  a nonabelian simple group,  $T\neq A_n$,  and $G$
acting primitively on  $\Omega$ (almost simple  case).
\end{theorem}

Now,  we consider the symmetric group $S_5$ and prove the following

\begin{theorem}Let $C$ be a cyclic subgroup of $A_{5}$, then the space $|\mathfrak{T}_{C}(S_{5})|$
is connected.
\end{theorem}

\begin{proof}Since $C$ is a cyclic subgroup of $A_{5}$, thus $C$ likes one of the following types:

(1) $C=1$;

(2) $C=\langle(123)\rangle$;

(3) $C=\langle(12345)\rangle$;

(4) $C=\langle(12)(34)\rangle$.

\textbf{Case 1.} $C=1$. By the Theorem 7.1, we can see that the maximal subgroup($\neq A_5$) of $S_5$ is the following types:
$$S_m\times S_k, ~ m+k=5;$$
$$AGL_1(5) \cap S_5.$$
If $M_1, M_2(\neq A_5)$ are maximal subgroups of $S_5$ and $M_1, M_2$ are type $S_m\times S_k, ~ m+k=5$,
then $M_1\cap M_2$ contains a subgroup which is isomorphic to $S_2$. And we can see one of type $AGL_1(5) \cap S_5$ as
$$\langle(12345), (2354)\rangle.$$
Let $M_3\cong S_4$ and act on the set $\{2,3,4,5\}$, then $M_3$ is a maximal subgroup of $S_5$ and
 $M_3\cap \langle(12345), (2354)\rangle\geq \langle(2354)\rangle\in \mathfrak{T}_{C}(S_{5}).$
So $|\mathfrak{T}_{C}(S_{5})|$ is connected.

\textbf{Case 2.} $C=\langle(123)\rangle$. By the Theorem 7.1, we can see that the maximal subgroup($\neq A_5$) of $S_5$ is the following type:
$$S_m\times S_k, ~ m+k=5.$$
If $M_1, M_2(\neq A_5)$ are maximal subgroups of $S_5$ and $M_1, M_2$ contain $C=\langle(123)\rangle$,
then $M_1\cap M_2$ contains a subgroup which is isomorphic to $S_3$ which acts on the set $\{1,2,3\}$.
So $|\mathfrak{T}_{C}(S_{5})|$ is connected.

\textbf{Case 3.} $C=\langle(12345)\rangle$. By the Theorem 7.1, we can see that the maximal subgroup($\neq A_5$) of $S_5$ is the following type:
$$AGL_1(5) \cap S_5.$$
We can see this type subgroup which contains $C$ is only
$$\langle(12345), (2354)\rangle.$$
So $|\mathfrak{T}_{C}(S_{5})|$ is connected.

\textbf{Case 4.} $C=\langle(12)(34)\rangle$. By the Theorem 7.1, we can see that the maximal subgroup($\neq A_5$) of $S_5$ is the following types:
$$S_m\times S_k, ~ m+k=5;$$
$$AGL_1(5) \cap S_5.$$
If $M_1, M_2(\neq A_5)$ are maximal subgroups of $S_5$ and $M_1, M_2$ are type $S_m\times S_k, ~ m+k=5$, then $M_1\cap M_2$ contains a subgroup which is isomorphic to $S_2$. And we can see one of type $AGL_1(5) \cap S_5$ as
$$\langle(25341), (1324)\rangle.$$
Let $M_3\cong S_4$ and act on the set $\{1,2,3,4\}$, then $M_3$ is a maximal subgroup of $S_5$ and
 $M_3\cap \langle(25341), (1324)\rangle\geq \langle(1324)\rangle\in \mathfrak{T}_{C}(S_{5}).$
So $|\mathfrak{T}_{C}(S_{5})|$ is connected.
\end{proof}

\begin{proposition}Let $G$ be the symmetric group $S_5$, and $N$ be the alternating group $A_5$ with $N\unlhd G$.
Then $m_{G, N}= 0$.
\end{proposition}

\begin{proof} By the Theorem 7.2, the Theorem 6.1 and the Lemma 6.2, we know that 
 $|\mathfrak{T}_{C}(G)|=1$ for each cyclic subgroup of $N\cong A_5$.
Hence by the Proposition 5.4,  we have 
$M_{G,N}'=0$. That is $m_{G,N}=0$.
\end{proof}

\begin{remark} In \cite{Bo3}, Bouc had proved that $S_n$ is a $B$-group. And by using \cite{Ba}, we also get that $S_n$
is a $B$-group.
\end{remark}

\textbf{ACKNOWLEDGMENTS}\hfil\break
The authors would like to thank Prof. S. Bouc for his numerous discussion in Beijing in Oct. 2014.
And the second author would like to thank  Prof. C. Broto for his constant encouragement in Barcelona in Spain.
Actually, the part of the idea of the Proposition 6.1 is due to Prof. C. Broto.


\begin{thebibliography}{13}



\bibitem{Ai} M. Aigner, Combinatorial theory, Springer-Verlag, Berlin, (1979)

\bibitem{AS}  M.  Aschbacher, L. Scott,  Maximal subgroups of finite groups, J. Alg. 92,
44-80 (1985)

\bibitem{Ba} M. Baumann, The composition factors of the functor of permutation modules,
J. Alg. 344, 284-295 (2011)

\bibitem{Bo1} S. Bouc, A conjecture on $B$-groups,
Math. Z. 274, 367-372 (2013)


\bibitem{Bo2} S. Bouc,  Biset functors for finite groups, LNM. No.1990

\bibitem{Bo3} S. Bouc, Foncteurs d'ensembles munis d'une double action, J. Alg. 183, 664-736 (1996)

\bibitem{CR} C. Curtis, I. Reiner, Methods of representation theory (II), Wiley, (1994)

\bibitem{Da} E. Dade, Endo-Permutation Modules over p-Groups,
Ann. Math. I, 107(1978); II, 108(1978)


\bibitem{Dr} A. Dress,
 A characterisation of solvable groups,
Math. Z.  110,  213-217 (1969)

\bibitem{DH} W. Dwyer, H. Henn,
 Homotopy theoretic methods in group cohomology, Advanced courses in mathematics. CRM Barcelona, Birkh$\mathrm{\ddot{a}}$user Verlag, (2001)

\bibitem{Gl} D. Gluck. Idempotent formula for the Burnside ring with applications to
the p-subgroup simplicial complex. Illinois J. Math., 25, 63-67 (1981)

\bibitem{G} W. Gorenstein, Finite groups, Chelsea, London (1968)

\bibitem{Ht} A. Hatcher, Algebraic topology, Cambridge Univ. Press (2002)

\bibitem{LPS} L. Liebeck, C. Praeger, J. Saxl, A classification  of the maximal  subgroups  of the finite alternating  and symmetric  groups, J. Alg. 111, 365-383 (1987)

\bibitem{M} J. Munkres, Elements of algebraic topology, Westview Press (1996)

\bibitem{XZ} X. Xu, J. Zhang, Bouc's conjecture on $B$-group, arXiv:1701.05985v1

\bibitem{Y} T. Yoshida,
 Idempotents of Burnside rings and Dress induction theorem,
J. Alg. 80, 90-105 (1983)

\end{thebibliography}
\end{document}